\documentclass[11pt,a4paper,reqno]{amsart}

\usepackage{amssymb,amsmath}
\usepackage{amsthm}
\usepackage{latexsym}
\usepackage[colorlinks=true,linktocpage=true,pagebackref=false, citecolor=black,linkcolor=black]{hyperref}
\usepackage{graphics}
\usepackage{euscript}
\usepackage{mathrsfs}
\usepackage[dvips]{curves}
\usepackage{tikz}
\usetikzlibrary{arrows,decorations.pathmorphing,backgrounds,positioning,fit,matrix}
\usepackage{comment}
\usepackage{multirow}
\usepackage{xcolor}
\usepackage[all]{xy}
\usepackage{caption,subcaption}

\newtheorem{thm}{Theorem}[section]

\newtheorem{prop}[thm]{Proposition}

\newtheorem{lem}[thm]{Lemma}

\newtheorem{prob}[thm]{Problem}

\theoremstyle{definition}

\newtheorem{define}[thm]{Definition}

\theoremstyle{remark}
\newtheorem{remark}[thm]{Remark}

\newtheorem{example}[thm]{Example}
\newtheorem{examples}[thm]{Examples}
\newtheorem{quest}[thm]{Question}

\numberwithin{equation}{section}
\renewcommand\thesubsection{\thesection.\Alph{subsection}}

 \newcommand{\J}{{\mathcal J}}

 \newcommand{\R}{{\mathbb R}}
 \newcommand{\C}{{\mathbb C}}

\newcommand{\sph}{{\mathbb S}} 
 \newcommand{\PP}{{\mathbb P}}

\newcommand{\pol}{{\EuScript K}}

\newcommand{\Pp}{{\EuScript P}}
\newcommand{\Ss}{{\EuScript S}}

\newcommand{\Ll}{{\EuScript L}}

\newcommand{\Qq}{{\EuScript Q}}

\newcommand\Rr{{\EuScript R}}

\newcommand{\Tt}{{\EuScript T}}
\newcommand{\Hh}{{\EuScript H}}

\newcommand{\Cc}{{\EuScript C}}

\newcommand{\tildebaja}{{\raise.17ex\hbox{$\scriptstyle\sim$}}}

\newcommand{\Int}{\operatorname{Int}}
\newcommand{\im}{\operatorname{Im}}

\newcommand{\id}{\operatorname{id}}

\newcommand{\zar}{\operatorname{zar}}

\newcommand{\Reg}{\operatorname{Reg}}

\newcommand{\mult}{\operatorname{mult}}

\newcommand{\x}{{\tt x}} \newcommand{\y}{{\tt y}}
 \renewcommand{\t}{{\tt t}}
\newcommand{\s}{{\tt s}}

\newcommand{\gta}{{\mathfrak a}}

\newcommand{\ol}{\overline}

\newcommand{\veps}{\varepsilon}

\begin{document}

\title[On regulous and regular images of Euclidean spaces]{On regulous and regular images of Euclidean spaces}

\author{Jos\'e F. Fernando}
\address{Departamento de \'Algebra, Facultad de Ciencias Matem\'aticas, Universidad Complutense de Madrid, 28040 MADRID (SPAIN)}
\email{josefer@mat.ucm.es}

\author{Goulwen Fichou}
\author{Ronan Quarez}
\address{Institut de Recherche Math\'ematiques de Rennes IRMAR (UMR 6625), Universit\'e de Rennes 1, Campus de Beaulieu, 35042 RENNES cedex (FRANCE)}
\email{goulwen.fichou@univ-rennes1.fr, ronan.quarez@univ-rennes1.fr}

\author{Carlos Ueno}
\address{Matematika Munkak\"oz\"oss\'eg, P\'ecsi Kod\'aly Zolt\'an Gimn\'azium, Dob\'o Istv\'an utca 35-37, 7629 P\'ECS (HUNGARY)}
\email{cuenjac@gmail.com}

\date{15/10/2017}

\thanks{First author supported by Spanish GRAYAS MTM2014-55565-P and Grupos UCM 910444. This article has been mainly written during several research stays of the first author in the Institut de Recherche Math\'ematiques de Rennes of the Universit\'e de Rennes 1. The first author would like to thank the department for the invitation and the very pleasant working conditions.}

\begin{abstract}--- \ 
In this work we compare the semialgebraic subsets that are images of regulous maps with those that are images of regular maps. Recall that a map $f:\R^n\to\R^m$ is \em regulous \em if it is a rational map that admits a continuous extension to $\R^n$. In case the set of (real) poles of $f$ is empty we say that it is \em regular map\em. We prove that if $\Ss\subset\R^m$ is the image of a regulous map $f:\R^n\to\R^m$, there exists a dense semialgebraic subset $\Tt\subset\Ss$ and a regular map $g:\R^n\to\R^m$ such that $g(\R^n)=\Tt$. In case $\dim(\Ss)=n$, we may assume that the difference $\Ss\setminus\Tt$ has codimension $\geq 2$ in $\Ss$. If we restrict our scope to regulous maps from the plane the result is neat: \em if $f:\R^2\to\R^m$ is a regulous map, there exists a regular map $g:\R^2\to\R^m$ such that $\im(f)=\im(g)$\em. In addition, we provide in the Appendix a regulous and a regular map $f,g:\R^2\to\R^2$ whose common image is the open quadrant $\Qq:=\{\x>0,\y>0\}$. These maps are much simpler than the best known polynomial maps $\R^2\to\R^2$ that have the open quadrant as their image.
\end{abstract}
\maketitle

\section{Introduction}\label{s1}

A map $f:=(f_1,\ldots,f_m):\R^n\to\R^m$ is \em regulous \em if its components are \em regulous functions\em, that is, $f_i=\frac{g_i}{h_i}$ is a rational function that admits a extension continuous extension to $\R^n$ and $g_i,h_i\in\R[\x]$ are relatively prime polynomials. By \cite[3.5]{fhmm} $\{h_i=0\}\subset\{g_i=0\}$ and the codimension of $\{h_i=0\}$ is $\geq2$. Consequently, the set of indeterminacy of $f$ is an algebraic set of codimension $\geq2$. In case $\{h_i=0\}$ is empty for each $i=1,\ldots,m$ we say that $f$ is a \em regular map \em whereas $f$ is \em polynomial \em if we may choose $h_1=\cdots=h_m=1$. Modern Real Algebraic Geometry was born with Tarski's article \cite{ta}, where it is proved that the image of a semialgebraic set under a polynomial map is a semialgebraic set. A subset $\Ss\subset\R^n$ is \em semialgebraic \em when it has a description by a finite boolean combination of polynomial equalities and inequalities, which we will call a \em semialgebraic \em description.

We are interested in studying what might be called the `inverse problem' to Tarski's result, that is, to represent semialgebraic sets as images of nice semialgebraic sets (in particular, Euclidean spaces, spheres, smooth algebraic sets, etc.) under semialgebraic maps that enjoy a kind of identity principle, that is, semialgebraic maps that are determined by its restriction to a small neighborhood of a point (for instance, polynomial, regular, regulous, Nash, etc.). In the 1990 \em Oberwolfach reelle algebraische Geometrie \em week \cite{g} Gamboa proposed: 

\begin{prob}\label{prob0}
To characterize the (semialgebraic) subsets of $\R^m$ that are either polynomial or regular images of Euclidean spaces. 
\end{prob}

During the last decade, the first and fourth authors jointly with Gamboa have attempted to understand better polynomial and regular images of $\R^n$ with the following target: 
\begin{itemize}
\item To find obstructions to be either polynomial or regular images \cite{fg1,fg2,fgu2,fu3}. 
\item To prove (constructively) that large families of semialgebraic sets with piecewise linear boundary (convex polyhedra, their interiors, their complements and the interiors of their complements) are either polynomial or regular images of Euclidean spaces \cite{fgu1,fgu3,fu1,fu2,fu4,u1,u2}. 
\end{itemize}

In \cite{fe1} appears a complete solution to Problem \ref{prob0} for the $1$-dimensional case, but the rigidity of polynomial and regular maps makes really difficult to approach Problem \ref{prob0} in its full generality. Taking into account the flexibility of Nash maps, Gamboa and Shiota discussed in 1990 the possibility of approaching the following variant of Problem \ref{prob0} (see \cite{g}). 

\begin{prob}\label{prob1}
To characterize the (semialgebraic) subsets of $\R^n$ that are Nash images of Euclidean spaces. 
\end{prob}

A \em Nash function on an open semialgebraic set \em $U\subset M$ is a semialgebraic smooth function on $U$. In \cite{fe2} the first author solves Problem \ref{prob1} and provides a full characterization of the semialgebraic subsets of $\R^m$ that are images under a Nash map on some Euclidean space. A natural alternative approach is to consider regulous maps, which are closer than Nash maps to regular maps but seem to be less rigid than the latter \cite{kz2,kz}. We propose to devise the following variant of Problem \ref{prob0}. 

\begin{prob}\label{prob2}
To characterize the (semialgebraic) subsets of $\R^n$ that are regulous images of Euclidean spaces. 
\end{prob}

By \cite[3.11]{fhmm} a map $f:\R^n\to\R^m$ is regulous if and only if there exists a finite sequence of blowings-up with non-singular centers $\phi:Z\to\R^n$ such that the composition $f\circ\phi:Z\to\R^m$ is a regular map. As $Z$ is a non-singular algebraic set, we deduce that $f(\R^n)=(f\circ\phi)(Z)$ is a pure dimensional semialgebraic subset of $\R^n$. In addition, as a regulous map on $\R$ is regular \cite[3.6]{fhmm}, we deduce that $f(\R^n)$ is connected by regular images of $\R$. We refer the reader to \cite{fhmm} for a foundational work concerning the ring of regulous functions on an algebraic set. 

Our experimental approach to the problem of characterizing the regulous images of Euclidean spaces suggests that regular images and regulous images of Euclidean spaces coincide. Although we have not been able to prove this result in its full generality our main result in this direction is the following.

\begin{thm}\label{main}
Let $f:\R^n\to\R^m$ be a regulous map and let $\Ss:=f(\R^n)$ be its image. Then, there exists a regular map $g:\R^n\to\R^m$ whose image $\Tt:= g(\R^n)\subset\Ss$ is dense in $\Ss$. In addition, if $\dim(\Ss)=n$, we may assume that the difference $\Ss\setminus\Tt$ has codimension $\geq 2$.
\end{thm}

The proof of the previous result is reduced to show the following one, which has interest by its own and concerns the representation of the complement of an algebraic subset of $\R^n$ of codimension $\geq2$ as a polynomial image of $\R^n$.

\begin{thm}\label{complement}
Let $X\subset\R^n$ be an algebraic set of codimension $\geq2$. Then the constructible set $\R^n\setminus X$ is a polynomial image of $\R^n$.
\end{thm}

Assume Theorem \ref{complement} is proved and let us see how the proof of Theorem \ref{main} arises.

\begin{proof}[Proof of Theorem \em \ref{main}]
As we have commented above, the set of indeterminacy $X$ of $f$ has codimension $\geq2$. By Theorem \ref{complement} there exists a polynomial map $h:\R^n\to\R^n$ such that $h(\R^n)=\R^n\setminus X$. Consider the composition $g:=f\circ h:\R^n\to\R^m$. We have 
$$
\Ss\setminus f(X)\subset\Tt:=g(\R^n)=f(\R^n\setminus X)\subset\Ss. 
$$
As $f$ is a continuous map and $\R^n\setminus X$ is dense in $\R^n$, we deduce that $\Tt$ is dense in $\Ss$. In addition, if $\dim(\Ss)=n$, we have $\Ss\setminus\Tt\subset f(X)$, so $\dim(\Ss\setminus\Tt)\leq\dim(f(X))$. By \cite[Thm.2.8.8]{bcr} the following inequalities
$$
\dim(f(X))\leq\dim(X)\leq n-2=\dim(\Ss)-2
$$
hold, so $\dim(\Ss\setminus\Tt)\leq\dim(\Ss)-2$, as required.
\end{proof}

As commented above we feel that the following question has a positive answer.
\begin{quest}\label{prob3}
Let $f:\R^n\to\R^m$ be a regulous map. Is there a regular map $g:\R^n\to\R^m$ such that $\im(f)=\im(g)$?
\end{quest}

As there are much more regulous maps than regular maps and they are less rigid that regular maps, we hope that a positive answer to the previous question will ease the solution to Problem \ref{prob0} in the regular case. If we restrict our target to regulous maps from the plane we devise the following net answer to Question \ref{prob3}. 
\begin{thm}\label{main2}
Let $f:\R^2\to\R^m$ be a regulous map. Then there exists a regular map $g:\R^2\to\R^m$ such that $\im(f)=\im(g)$.
\end{thm}

The proof of the previous result is reduced to show the following lemma, which has interest by its own and concerns an `alternative' resolution of the indeterminacy of a locally bounded rational function on $\R^2$ to make it regular, adapted to our situation, in which one does not care on the cardinality of the fibers (that are however generically finite).

\begin{lem}\label{resol0} 
Let $f:\R^2\to\R$ be a locally bounded rational function on $\R^2$. Then there exists a generically finite surjective regular map $\phi:\R^2\to\R^2$ such that $f\circ\phi$ is a regular function on $\R^2$.
\end{lem}

Assume Lemma \ref{resol0} and let us see how the proof of Theorem \ref{main2} arises.

\begin{proof} 
Let $f:=(f_1,\cdots,f_m):\R^2\to\R^m$ be a regulous maps. By Lemma \ref{resol0} one can find a surjective regular map $\phi_1:\R^2\to\R^2$ such that $f_1\circ\phi_1$ is regular. Consider the regulous function $g_2:=f_2\circ\phi_1$. Again by Lemma \ref{resol0} one can find a surjective regular map $\phi_2:\R^2\to\R^2$ such that $g_2\circ\phi_2$ is regular. Clearly, the composition $f_1\circ\phi_1\circ\phi_2$ is also regular. We proceed inductively with $f_3,\ldots,f_m$ and obtain surjective regular maps $\phi_k:\R^2\to\R^2$ such that $f_k\circ\phi_1\circ\cdots\circ\phi_k$ is a regular function for $k=1,\ldots,m$. At the end we have a surjective regular map $\phi:=\phi_1\circ\cdots\circ\phi_m:\R^2\to\R^2$ such that $g:=f\circ\phi$ is a regular map on $\R^2$. As $\phi$ is surjective, we have $g(\R^2)=(f\circ\phi)(\R^2)=f(\R^2)$, as required.
\end{proof}

\subsection*{Structure of the article.}
The article is organized as follows. In Section \ref{s2} we prove Theorem \ref{complement} while in Section \ref{s3} we prove an improved version of Lemma \ref{resol0}. Our argument fundamentally relies on the properties of the so-called double oriented blowings-up. This tools allows us to reformulate Lemma \ref{resol0} using regular maps satisfying the arc lifting property (see Theorem \ref{SurjectiveArcLift}), which is commonly used in the analytic setting to study singularities \cite{fp}. Finally in Appendix \ref{s4} we propose a regulous and a regular map $f,g:\R^2\to\R^2$ whose common image is the open quadrant $\Qq:=\{\x>0,\y>0\}$. Theses maps are much more simpler than the best known polynomial maps $\R^2\to\R^2$ that have $\Qq$ as their image \cite{fg1,fu5,fgu4}. In addition, the verification that the images of $f,g$ is $\Qq$ is quite straightforward and does not require the enormous effort needed for the known polynomial maps used in \cite{fg1,fu5,fgu4} to represent $\Qq$. Recall that the systematic study of the polynomial and regular images began in 2005 with the first solution to the open quadrant problem \cite{fg1}.

\section{Complement of an algebraic set of codimension at least $2$}\label{s2}

The purpose of this section is to prove Theorem \ref{complement}. The proof is conducted in several technical steps and some of them have interest by their own. 

\subsection{Finite projections of complex algebraic sets}
Let $X\subset\R^n$ be a (real) algebraic set and let $I_\R(X)$ be the ideal of those polynomials $f\in\R[\x]$ such that $f(x)=0$ for all $x\in X$. The zero-set $\widetilde{X}\subset\C^n$ of $I_\R(X)\C[\x]$ is a complex algebraic set that is called the \em complexification of $X$\em. The ideal $I_\C(\widetilde{X})$ of those polynomials $F\in\C[\x]$ such that $F(z)=0$ for all $z\in\widetilde{X}$ coincides with $I_\R(X)\C[\x]$. Consequently $\widetilde{X}$ is the smallest complex algebraic subset of $\C^n$ that contains $X$ and
$$
\C[\x]/I_\C(\widetilde{X})\cong(\R[\x]/I_\R(X))\otimes\C.
$$
In particular, the real dimension of $X$ coincides with the complex dimension of $\widetilde{X}$. A key result in this section is the following result concerning projections of complex algebraic sets \cite[Thm. 2.2.8]{jp}. The zero set in $\C^n$ of an ideal $\gta\subset\C[\x]$ will be denoted ${\mathcal Z}(\gta)$.

\begin{thm}[Projection Theorem]\label{pt} 
Let $\gta\subset\C[\x]$ be a non-zero ideal. Suppose that $\gta$ contains a polynomial $F$ that is regular with respect to $\x_n$, that is, $F:=\x_n^m+a_{m-1}(\x')\x_n^{m-1}+\cdots+a_0(\x')$ for some polynomials $a_i\in\C[\x']:=\C[\x_1,\ldots,\x_{n-1}]$. Let $\Pi:\C^n\to\C^{n-1}$ be the projection $(x_1,\ldots,x_n)\mapsto(x_1,\dots,x_{n-1})$ onto the first $n-1$ coordinates of $\C^n$. Define $\gta':=\gta\cap\C[\x']$ and let $X':={\mathcal Z}(\gta')$ be the its zero set. Then $\Pi(X)=X'$ is a complex algebraic subset of $\C^{n-1}$.
\end{thm}

\subsection{Recovering an algebraic set from its projections}
Denote 
$$
\pi:\R^n\to\R^{n-1},\ (x_1,\ldots,x_n)\to(x_1,\ldots,x_{n-1})
$$ 
the projection onto the first $n-1$ coordinates of $\R^n$. Pick $\vec{v}:=(v_1,\ldots,v_n)\in\R^n\setminus\{\x_n=0\}$ and consider the isomorphism 
\begin{equation}\label{iso}
\phi_{\vec{v}}:\R^n\to\R^n,\ (x_1,\ldots,x_n)\mapsto(x_1-\tfrac{v_1}{v_n}x_n,\ldots,x_{n-1}-\tfrac{v_{n-1}}{v_n}x_n,\tfrac{1}{v_n}x_n)
\end{equation} 
that keeps fixed the hyperplane $H_0:=\{\x_n=0\}$ and maps the vector $\vec{v}$ to the vector $\vec{\tt e}_n:=(0,\ldots,0,1)$. Its inverse map is
$$
\phi_{\vec{v}}^{-1}:\R^n\to\R^n,\ (y_1,\ldots,y_n)\mapsto(y_1+v_1y_n,\ldots,y_{n-1}+v_{n-1}y_n,v_ny_n).
$$
Consider the projection
\begin{multline}\label{proj1}
\pi_{\vec{v}}:=\pi\circ\phi_{\vec{v}}:\R^n\to\R^{n-1}\times\{0\}\equiv\R^{n-1},\\ 
(x_1,\ldots,x_n)\mapsto(x_1-\tfrac{v_1}{v_n}x_n,\ldots,x_{n-1}-\tfrac{v_{n-1}}{v_n}x_n,0)\equiv(x_1-\tfrac{v_1}{v_n}x_n,\ldots,x_{n-1}-\tfrac{v_{n-1}}{v_n}x_n)
\end{multline}
of $\R^n$ onto $\R^{n-1}$ (identified with $\{\x_n=0\}$) in the direction of $\vec{v}$. We have 
$$
\pi_{\vec{v}}(y_1+\tfrac{v_1}{v_n}\lambda y_n,\ldots,y_{n-1}+\tfrac{v_{n-1}}{v_n}\lambda y_n,\lambda y_n)=(y_1,\ldots,y_{n-1})
$$
for each $\lambda\in\R$ (in particular, for $\lambda=v_n$). Let $X\subset\R^n$ be a real algebraic set of dimension $d\leq n-2$, let $\widetilde{X}$ be its complexification and let $\gta:=I_\C(\widetilde{X})=I_\R(X)\C[\x]$. Let $F\in I_\R(X)\setminus\{0\}$ and write $F:=F_0+F_1+\cdots+F_m$ as the sum of its homogeneous components. If $F_m(\vec{v})\neq0$, then $F(y_1+v_1y_n,\ldots,y_{n-1}+v_{n-1}y_n,v_ny_n)$ is regular with respect to $y_n$. We extend $\pi_{\vec{v}}$ to $\Pi_{\vec{v}}:\C^n\to\C^{n-1}$ and observe that, if $F_m(\vec{v})\neq0$, then $\Pi_{\vec{v}}(\widetilde{X})$ is by Theorem \ref{pt} an algebraic subset of $\C^{n-1}$ of dimension $\leq d$. Consequently, the Zariski closure $\ol{\pi_{\vec{v}}(X)}^{\zar}$ of $\pi_{\vec{v}}(X)$ is contained in $\Pi_{\vec{v}}(\widetilde{X})\cap\R^{n-1}$, which is a real algebraic subset of $\R^{n-1}$ of dimension $\leq d$. 

\paragraph{}Let $p\in\R^n\setminus X$ and consider the algebraic cone $\Cc\subset\R^n$ of vertex $p$ and basis $X$.
As $X$ has real dimension $d\leq n-2$, the real algebraic set $\Cc$ has (real) dimension $d+1\leq n-1$. Let $\widetilde{\Cc}\subset\C^n$ be the complexification of $\Cc$, which is a complex algebraic set of (complex) dimension $d+1$. As $X\subset\Cc$, the inclusion $\widetilde{X}\subset\widetilde{\Cc}$ holds. In addition, as $\Cc$ is a cone, then $\widetilde{\Cc}$ is also a cone. Denote $\vec{\Cc}:=\{\vec{v}\in\R^n:\ \vec{v}=\overrightarrow{px},\ x\in\Cc\}$ and $\vec{\widetilde{\Cc}}:=\{\vec{w}\in\C^n:\ \vec{w}=\overrightarrow{pz},\ z\in\widetilde{\Cc}\}$.

\paragraph{}\label{point} Pick a vector $\vec{v}\in\R^n\setminus(\vec{\Cc}\cup\{F_m=0\})$. We claim: $\pi_{\vec{v}}(p)\not\in\ol{\pi_{\vec{v}}(X)}^{\zar}$. 
\begin{proof}
As $\ol{\pi_{\vec{v}}(X)}^{\zar}\subset\Pi_{\vec{v}}(\widetilde{X})\cap\R^{n-1}$, it is enough to check that $\pi_{\vec{v}}(p)\not\in\Pi_{\vec{v}}(\widetilde{X})$. If $\pi_{\vec{v}}(p)\in\Pi_{\vec{v}}(\widetilde{X})$, there exists $z\in\widetilde{X}$ such that $\Pi_{\vec{v}}(z)=\pi_{\vec{v}}(p)$. Thus, there exists $\lambda\in\R$ and $\mu\in\C$ such that 
$$
p+\lambda\vec{v}=\pi_{\vec{v}}(p)=\Pi_{\vec{v}}(z)=z+\mu\vec{v}.
$$
Thus, $(\lambda-\mu)\vec{v}=\overrightarrow{pz}\in\vec{\widetilde{\Cc}}$. As $p\not\in X$, it holds $p\neq z$, so $\lambda\neq\mu$ and $\vec{v}\in\vec{\widetilde{\Cc}}\cap\R^n=\vec{\Cc}$, which is a contradiction. Thus, $\pi_{\vec{v}}(p)\not\in\ol{\pi_{\vec{v}}(X)}^{\zar}$, as claimed.
\end{proof}

\begin{lem}\label{recover}
Let $X\subset\R^n$ be a real algebraic set of dimension $d\leq n-2$ and let $\Omega\subset\R^n\setminus\{\x_n=0\}$ be a non-empty open set. Then, there exist $r\leq n+1$ vectors $\vec{v}_1,\ldots,\vec{v}_r\in\Omega$ such that
$$
X=\bigcap_{i=1}^r\pi_i^{-1}(\ol{\pi_j(X)}^{\zar})
$$
where $\pi_i:=\pi_{\vec{v}_i}$ for $i=1,\ldots,r$. 
\end{lem}
\begin{proof} 
Given an algebraic set $Z\subset\R^n$ define $e(Z):=\dim(Z\setminus X)$, that is, the dimension of the constructible set $Z\setminus X$. If $Z^1,\ldots,Z^k$ are the irreducible components of $Z$, then $e(Z)$ is either equal to $-1$ if $Z^j\subset X$ for $1\leq j\leq k$ or the maximum of the dimensions of the irreducible components of $Z$ that are not contained in $X$ otherwise.

Pick a vector $\vec{v}_1\in\Omega$ and denote $\pi_1:=\pi_{\vec{v}_1}$. Let $Y_1^1,\dots,Y_1^s$ be the irreducible components of the algebraic set $Y_1:=\pi_1^{-1}\big(\ol{\pi_1(X)}^{\zar}\big)$, which has dimension $\leq d+1$. If $e(Y_1)=-1$, we have $X=Y_1$ and we are done. Assume $e(Y_1)\geq0$. For each $Y_1^j$ not contained in $X$ pick a point $p_j\in Y_1^j\setminus X$. By \ref{point} there exists a vector $\vec{v}_2\in\Omega$ such that each $p_j\not\in Y_2:=\pi_2^{-1}\big(\ol{\pi_2(X)}^{\zar}\big)$, where $\pi_2:=\pi_{\vec{v}_2}$. Let $T$ be an irreducible component of $Y_1\cap Y_2$. We claim: $\dim(T)<e(Y_1)$. 

As $T$ is not contained in $X$ but $T\subset Y_1$, there exists an irreducible component $Y_1^j$ of $Y_1$ not contained in $X$ such that $T\subset Y_1^j$. As $p_j\not\in Y_2$, we have $\dim(T)\leq\dim(Y_1^j\cap Y_2)<\dim(Y_1^j)\leq e(Y_1)$, as claimed. 

Consequently, $e(Y_1\cap Y_2)$ is strictly smaller than $e(Y_1)$. If $e(Y_1\cap Y_2)\geq0$, we pick a point $q_j\in Y_{12}^j\setminus X$ for each irreducible component $Y_{12}^j$ of $Y_1\cap Y_2$ not contained in $X$ (and indexed $j=1,\ldots,\ell$). By \ref{point} there exists a vector $\vec{v}_3\in\Omega$ such that each $q_j\not\in Y_3:=\pi_3^{-1}\big(\ol{\pi_3(X)}^{\zar}\big)$, where $\pi_3:=\pi_{\vec{v}_3}$. Again, this implies that $e(Y_1\cap Y_2\cap Y_3)<e(Y_1\cap Y_2)$. 

We repeat the process $r\leq d+3\leq n+1$ times to find vector $\vec{v}_1,\ldots,\vec{v_r}\in\Omega$ such that $e(Y_1\cap\cdots\cap Y_r)=-1$ where $Y_i:=\pi_i^{-1}(\ol{\pi_i(X)}^{\zar})$ and $\pi_i:=\pi_{\vec{v}_i}$. Consequently, $X=\bigcap_{j=1}^rY_j$, as required.
\end{proof}

\begin{lem}\label{cod2}
Let $X\subset\R^n$ be an algebraic set of codimension $\geq1$. Then there exists a polynomial diffeomorphism $\phi:\R^n\to\R^n$ such that $\phi(X)\subset\{-\x_{n-1}<\x_n<\x_{n-1}\}$.
\end{lem}
\begin{proof}
Let $\widehat{X}$ be the Zariski closure of $X$ in the projective space $\R\PP^n$. As $\widehat{X}$ has codimension $\geq1$, we may assume that $(0:\cdots:0:1:0)\not\in\widehat{X}$. The sets $\{\x_0^2+\cdots+\x_{n-2}^2+\x_n^2\leq\veps^2\x_{n-1}^2\}$ for $\veps>0$ constitute a basis of neighborhoods of $(0:\cdots:0:1:0)$ in $\R\PP^n$. Let $0<\veps<1$ be such that $\widehat{X}\cap\{\x_0^2+\cdots+\x_{n-2}^2+\x_n^2\leq\veps^2\x_{n-1}^2\}=\varnothing$. As $\veps^2\x_{n-1}^2\geq-\veps\x_{n-1}\x_0-\tfrac{1}{4}\x_0^2$ (because $\veps^2\x_{n-1}^2+\veps\x_{n-1}\x_0+\tfrac{1}{4}\x_0^2=(\veps\x_{n-1}+\tfrac{1}{2}\x_0)^2$), we have
$$
\{\x_0^2+\cdots+\x_{n-2}^2+\x_n^2\leq-\veps\x_{n-1}\x_0-\tfrac{1}{4}\x_0^2\}\subset\{\x_0^2+\cdots+\x_{n-2}^2+\x_n^2\leq\veps^2\x_{n-1}^2\}.
$$
Consequently, $X\subset\{\frac{5}{4}+\x_1^2+\cdots+\x_{n-2}^2+\x_n^2>-\veps\x_{n-1}\}$. Consider the polynomial diffeomorphism
$$
\phi:\R^n\to\R^n,\ x:=(x_1,\ldots,x_n)\mapsto(x_1,\ldots,x_{n-2},x_{n-1}+\tfrac{2}{\veps}(\tfrac{5}{4}+x_1^2+\cdots+x_{n-2}^2+x_n^2),x_n)
$$
whose inverse
$$
\phi^{-1}:\R^n\to\R^n,\ x:=(x_1,\ldots,x_n)\mapsto(x_1,\ldots,x_{n-2},x_{n-1}-\tfrac{2}{\veps}(\tfrac{5}{4}+x_1^2+\cdots+x_{n-2}^2+x_n^2),x_n)
$$
is also polynomial. We have 
$$
\phi(X)\subset\{\x_{n-1}\geq\tfrac{1}{\veps}(\tfrac{5}{4}+\x_1^2+\cdots+\x_{n-2}^2+\x_n^2)\}.
$$
As $\veps<1$, 
$$
\{\x_{n-1}\geq\tfrac{1}{\veps}(\tfrac{5}{4}+\x_1^2+\cdots+\x_{n-2}^2+\x_n^2)\}\subset\{-\x_{n-1}<\x_n<\x_{n-1}\},
$$
as required.
\end{proof}

\subsection{Proof of Theorem \ref{complement}}
We will represent $\R^n\setminus X$ as the image of the composition of finitely many polynomial maps $f_j:\R^n\to\R^n$ whose images contain constructible sets $\R^n\setminus Y_j$ such that $X=\bigcap_jY_j$. The proof is conducted in several steps:

\noindent{\em Step 1}. \em Initial preparation\em. By Lemma \ref{cod2} we assume $X\subset\Int(\pol)$ where $\pol:=\{-\x_{n-1}\le\x_n\le\x_{n-1}\}$. Denote the projection onto the first $n-1$ coordinates by 
$$
\pi:\R^n\to\R^{n-1},\ (x_1,\ldots,x_n)\mapsto(x_1,\ldots,x_{n-1}).
$$
For each vector $\vec{v}:=(v_1,\ldots,v_n)\in\R^n\setminus\{\x_n=0\}$ consider the isomorphism $\phi_{\vec{v}}$ that keeps fixed the plane $H_0:=\{\x_n=0\}$ and maps the vector $\vec{v}$ to the vector $\vec{\tt e}_n$ (see \eqref{iso}) and let $\pi_{\vec{v}}:=\pi\circ\phi_{\vec{v}}$ be the projection of $\R^n$ onto $\R^{n-1}$ (identified with $\{\x_n=0\}$) in the direction of $\vec{v}$ (see \eqref{proj1}). We have
$$
\phi_{\vec{v}}(\pol)=\{\x_{n-1}-(v_n-v_{n-1})\x_n\geq0,\x_{n-1}+(v_n+v_{n-1})\x_n\geq0\}.
$$
If $\lambda:=v_n-v_{n-1}>0$ and $\mu:=v_n+v_{n-1}>0$, then 
$$
\pi_{\vec{v}}(\pol\setminus\{\x_{n-1}=0,\x_n=0\})=\pi(\phi_{\vec{v}}(\pol)\setminus\{\x_{n-1}=0,\x_n=0\})=\{\x_{n-1}>0\}\subset\R^{n-1}. 
$$
Consider the open set $\Omega:=\{\x_n-\x_{n-1}>0,\x_n+\x_{n-1}>0\}\subset\R^n\setminus\{\x_n=0\}$. By Lemma \ref{recover} there exist vectors $\vec{v}_1,\dots,\vec{v}_r\in\Omega$ such that:
$$
X=\bigcap_{j=1}^r\pi_j^{-1}(\ol{\pi_j(X)}^{\zar})
$$
where $\pi_j:=\pi_{\vec{v}_j}$ for $j=1,\ldots,r$. In particular, each set $\pi_j^{-1}(\ol{\pi_j(X)}^{\zar})$ is an algebraic subset of $\R^n$ that contains $X$. For each $j=1,\ldots,r$ denote $\phi_j:=\phi_{\vec{v}_j}$. 

\noindent{\em Step 2}. We claim:\em There exists a polynomial map $f_1:\R^n\to\R^n$ such that
\begin{equation}\label{2}
\Ss_1:=f_1(\R^n\setminus\Int(\pol))=\R^n\setminus(\Int(\pol)\cap\pi_1^{-1}(\ol{\pi_1(X)}^{\zar}))\subset\R^n\setminus X.
\end{equation}
\em

Define $\pol_1:=\phi_1(\pol)$ and $X_1:=\phi_1(X)$. Write $\pol_1:=\{\x_{n-1}-\lambda_1\x_n\geq0,\x_{n-1}+\mu_1\x_n\geq0\}$ for some real numbers $\lambda_1,\mu_1>0$ and consider the polynomial map
$$
f_1':\R^n\to\R^n,\ x:=(x_1,\ldots,x_n)\mapsto (x_1,\ldots,x_{n-1},x_n(1-H_1(x_{n-1},x_n)^2G_1^2(x')))
$$
where $H_1:=(\x_{n-1}-\lambda_1\x_n)(\x_{n-1}+\mu_1\x_n)\in\R[\x_{n-1},\x_n]$ and $G_1\in\R[\x']:=\R[\x_1,\ldots,\x_{n-1}]$ is a polynomial equation of $\ol{\pi(X_1)}^{\zar}$. We claim: 
$$
f_1'(\R^n\setminus\Int(\pol_1))=\R^n\setminus(\Int(\pol_1)\cap(\pi^{-1}(\ol{\pi(X_1)}^{\zar}))).
$$

Pick a point $x':=(x_1,\ldots,x_{n-1})\in\R^{n-1}$ and let $\ell_{x'}:=\{(x',t):\ t\in\R\}$ be the line through $(x',0)$ parallel to $\vec{\tt e}_n$. Consider the polynomial $Q_{x'}(\x_n):=\x_n(1-H_1(x',\x_n)^2G_1^2(x'))\in\R[\x_n]$. We distinguish two cases:

\par$\bullet$ If $x'\not\in\ol{\pi(X_1)}^{\zar}$, then $Q_{x'}(\x_n)$ is a polynomial of degree five and negative leading coefficient. Let $a_{x'}\leq b_{x'}$ be real numbers such that 
$$
\ell_{x'}\setminus\Int(\pol_1)=\{x'\}\times((-\infty,a_{x'}]\cup[b_{x'},+\infty)).
$$
As $\lim_{x_n\to\pm\infty}Q_{x'}(x_n)=\mp\infty$, $Q_{x'}(a_{x'})=a_{x'}$ and $Q_{x'}(b_{x'})=b_{x'}$, we have 
\begin{align*}
[a_{x'},+\infty)&\subset Q_{x'}((-\infty,a_{x'}]),\\
(-\infty,b_{x'}]&\subset Q_{x'}([b_{x'},+\infty)).
\end{align*}
Thus, $\R=(-\infty,b_{x'}]\cup[a_{x'},+\infty)\subset Q_{x'}(\ell_{x'}\setminus\Int(\pol_1))\subset\R$, so $f_1'(\ell_{x'}\setminus\Int(\pol_1))=\ell_{x'}$.

\par$\bullet$ If $x'\in\ol{\pi(X_1)}^{\zar}$, then $Q_{x'}(\x_n)=\x_n$ and $f_1'(\ell_{x'}\setminus\Int(\pol_1))=\ell_{x'}\setminus\Int(\pol_1)$. 

Putting all together we deduce 
\begin{multline*}
f_1'(\R^n\setminus\Int(\pol_1))=\bigcup_{x'\in\R^{n-1}}f_1'(\ell_{x'}\setminus\Int(\pol_1))\\
=\bigcup_{x'\not\in\ol{\pi(X_1)}^{\zar}}\ell_{x'}\cup\bigcup_{x'\in\ol{\pi(X_1)}^{\zar}}(\ell_{x'}\setminus\Int(\pol_1))=\R^n\setminus(\Int(\pol_1)\cap(\pi^{-1}(\ol{\pi(X_1)}^{\zar}))).
\end{multline*}
As $X_1\subset\Int(\pol_1)\cap\pi^{-1}(\ol{\pi(X_1)}^{\zar})$, we have
\begin{equation}\label{1}
f_1'(\R^n\setminus\Int(\pol_1))=\R^n\setminus(\Int(\pol_1)\cap\pi^{-1}(\ol{\pi(X_1)}^{\zar}))\subset\R^n\setminus X_1.
\end{equation}

Recall that $\pi_1=\pi\circ\phi_1$ and define $f_1:=\phi_1^{-1}\circ f_1'\circ\phi_1$. If we apply $\phi_1^{-1}$ to \eqref{1}, we deduce that $f_1$ satisfy condition \eqref{2} above.

\noindent{\em Step 3}. We claim: \em There exists a polynomial map $f_2:\R^n\to\R^n$ such that
\begin{equation}\label{4}
\R^n\setminus(\Int(\pol)\cap\pi_1^{-1}(\ol{\pi_1(X)}^{\zar})\cap(\pi_2^{-1}(\ol{\pi_2(X)}^{\zar})))\subset\Ss_2:=f_2(\Ss_1)\subset\R^n\setminus X.
\end{equation}
\em

Define $\pol_2:=\phi_2(\pol)$ and $X_2:=\phi_2(X)$. Write $\pol_2:=\{\x_{n-1}-\lambda_2\x_n\geq0,\x_{n-1}+\mu_2\x_n\geq0\}$ for some real numbers $\lambda_2,\mu_2>0$ and consider
$$
f_2':\R^n\to\R^n,\ x:=(x_1,\ldots,x_n)\mapsto (x_1,\ldots,x_{n-1},x_n(1-H_2^2(x_{n-1},x_n)G_2^2(x'))),
$$
where $H_2:=(\x_{n-1}-\lambda_2\x_n)(\x_{n-1}+\mu_2\x_n)\in\R[\x_{n-1},\x_n]$ and $G_2\in\R[\x']$ is a polynomial equation of $\ol{\pi(X_2)}^{\zar}$. Let us check:
\begin{equation}\label{3}
\R^n\setminus(\Int(\pol_2)\cap\pi^{-1}(\ol{\pi(X_2)}^{\zar})\cap\phi_2(\pi_1^{-1}(\ol{\pi_1(X)}^{\zar})))\subset f_2'(\phi_2(\Ss_1))\subset\R^n\setminus X_2.
\end{equation}

Proceeding as in Step 1 for the vector $\vec{v}_1$ we have 
$$
f_2'(\R^n\setminus\Int(\pol_2))=\R^n\setminus(\Int(\pol_2)\cap\pi^{-1}(\ol{\pi(X_2)}^{\zar})).
$$
By \eqref{2} $\R^n\setminus(\Int(\pol)\cap\pi_1^{-1}(\ol{\pi_1(X)}^{\zar}))=\Ss_1$, so
$$
\R^n\setminus(\Int(\pol_2)\cap\phi_2(\pi_1^{-1}(\ol{\pi_1(X)}^{\zar})))=\phi_2(\Ss_1).
$$
As $\pi^{-1}(\ol{\pi(X_2)}^{\zar})=\{G_2=0\}$, we have $f_2'|_{\pi^{-1}(\ol{\pi(X_2)}^{\zar})}=\id_{\pi^{-1}(\ol{\pi(X_2)}^{\zar})}$. Thus, $f_2'|_{X_2}=\id_{X_2}$ and $\pi^{-1}(\ol{\pi(X_2)}^{\zar})\setminus(\Int(\pol_2)\cap\phi_2(\pi_1^{-1}(\ol{\pi_1(X)}^{\zar}))\subset f_2'(\phi_2(\Ss_1))$. Consequently,
\begin{multline*}
\R^n\setminus(\Int(\pol_2)\cap\pi^{-1}(\ol{\pi(X_2)}^{\zar})\cap\phi_2(\pi_1^{-1}(\ol{\pi_1(X)}^{\zar})))=(\R^n\setminus(\Int(\pol_2)\cap\pi^{-1}(\ol{\pi(X_2)}^{\zar})))\\
\cup(\pi^{-1}(\ol{\pi(X_2)}^{\zar})\setminus(\Int(\pol_2)\cap\phi_2(\pi_1^{-1}(\ol{\pi_1(X)}^{\zar})))\subset f_2'(\phi_2(\Ss_1)).
\end{multline*}
As $\Ss_1\subset\R^n\setminus X$, we have $\phi_2(\Ss_1)\subset\R^n\setminus X_2$, so $f_2'(\phi_2(\Ss_1))\subset f_2'(\R^n\setminus X_2)$. 

Let us check: $f_2'^{-1}(X_2)=X_2$. Once this is proved, $f_2'(\R^n\setminus X_2)\subset\R^n\setminus X_2$ and the remaining part of equation \eqref{3} holds. 

Pick $y:=(y',y_n)\in\R^n$ such that $f_2'(y)\in X_2$. Then $y'=\pi(f_2'(y))\in\pi(X_2)$, so $G_2(y')=0$ and $y=f_2'(y)\in X_2$. Thus, $f_2'^{-1}(X_2)\subset X_2$. As $f_2'|_{X_2}=\id_{X_2}$, we deduce $f_2'^{-1}(X_2)=X_2$.

Recall that $\pi_2=\pi\circ\phi_2$ and define $f_2:=\phi_2^{-1}\circ f_2'\circ\phi_2$. If we apply $\phi_2^{-1}$ to \eqref{3}, we deduce that $f_2$ satisfy condition \eqref{4} above.

\noindent{\em Step 4}. We proceed similarly with the remaining vectors $\vec{v}_j$ for $j=1,\ldots,r$ to obtain polynomial maps $f_j:\R^n\to\R^n$ such that
$$
\R^n\setminus\Big(\Int(\pol)\cap\bigcap_{i=1}^j\pi_i^{-1}(\ol{\pi_i(X)}^{\zar})\Big)\subset\Ss_j:=f_j(\Ss_{j-1})\subset\R^n\setminus X
$$
where $\Ss_0:=\R^n\setminus\Int(\pol)$.

If we write $X_j:=\phi_j(X)$ and $\pol_j:=\phi_j(\pol)=\{\x_{n-1}-\lambda_j\x_n\geq0,\x_{n-1}+\mu_j\x_n\geq0\}$ for some real numbers $\lambda_j,\mu_j>0$, the sought polynomial map $f_j$ is the composition $f_j:=\phi_j^{-1}\circ f_j'\circ\phi_j$ where
$$
f_j':\R^n\to\R^n,\ x:=(x_1,\ldots,x_n)\mapsto (x_1,\ldots,x_{n-1},x_n(1-H_j^2(x_{n-1},x_n)G_j^2(x'))),
$$
$H_j:=(\x_{n-1}-\lambda_j\x_n)(\x_{n-1}+\mu_j\x_n)\in\R[\x_{n-1},\x_n]$ and $G_j\in\R[\x']$ is a polynomial equation of $\ol{\pi(X_j)}^{\zar}$.

\noindent{\em Step 5}. {\em Conclusion}. For $j=r$ we have
$$
\R^n\setminus X=\R^n\setminus\Big(\Int(\pol)\cap\bigcap_{j=1}^r\pi_j^{-1}(\ol{\pi_j(X)}^{\zar})\Big)\subset\Ss_r\subset\R^n\setminus X.
$$
Thus, $\R^n\setminus X=\Ss_r=(f_r\circ\cdots\circ f_1)(\R^n\setminus\Int(\pol))$. By \cite[Thm.1.3]{fu4} the semialgebraic set $\R^n\setminus\Int(\pol)$ is a polynomial image of $\R^n$, so $\R^n\setminus X$ is a polynomial image of $\R^n$, as required.
\qed

\section{Resolution of the indeterminacy using double oriented blowings-up}\label{s3}

In this section we prove Lemma \ref{resol0}, that is: \em we can make regular each regulous function (or even each locally bounded rational function) on $\R^2$ after composing it with a suitable generically finite surjective regular map $\phi:\R^2\to\R^2$\em. It is know in general that a regulous function can be made regular after composition with a finite sequence of algebraic blowings-up, but of course the source space is no longer the plane. We prove in this section that we can achieve the desired generically finite surjective regular map via a finite composition of double oriented blowings-up. Let us begin with some examples.

\begin{examples}\label{ex}
(1) Consider the regulous function $f:\R^2\to\R$ given by $f(x,y):=\frac{x^3}{x^2+y^2}$. The polynomial map $\varphi:\R^2\to\R^2$ defined by $\varphi(u,v):=(u(u^2+v^2),v(u^2+v^2))$ is surjective, and $(f\circ\varphi)(u,v)=u^3$ is polynomial. Note that $\varphi$ is a polynomial homeomorphism.

(2) Consider the rational function on $\R^2$ given by $g(x,y):=\frac{x}{x^2+y^2}$ defined and continuous on $\R^2\setminus\{(0,0)\}$. The polynomial map $\phi:\R^2\to\R^2$ defined by $\phi(u,v):=(v^3(uv+1),v(uv-1))$ is surjective. 

Actually, if $\phi(u,v)=(a,b)\neq (0,0)$, then $v\neq 0$ and $uv=1+b/v$. Thus, it is enough to prove that there exists $v\in\R\setminus\{0\}$ satisfying $2v^3+bv^2-a=0$. This is true if $a\neq0$ because the previous equation in $v$ has odd degree whereas in case $a=0$ we take $v=-\frac{b}{2}$. 

In addition, the composition $(g\circ\phi)(u,v)=\frac{v(uv+1)}{v^4(uv+1)^2+(uv-1)^2}$ is regular. However, the surjective regular map $\phi$ is not proper because $\phi^{-1}((0,0))=\{v=0\}$.
\end{examples} 

\begin{remark}
The previous surjective regular map $\phi:\R^2\to\R^2$ has the form 
$$
(u,v)\mapsto (P^k(u,v)Q(u,v),P(u,v)R(u,v))
$$ 
where $k\geq1$ is an odd integer and $P,Q,R\in\R[{\tt u},{\tt v}]$ are polynomials such that $\{PQ=0\}\cap\{R=0\}=\varnothing$. The latter condition enables us to soften $g$ by making that the denominator has empty zero-set after increasing numerator's multiplicity. Another (more complicate!) possible surjective regular map such that $g\circ\phi$ is regular could be
$$
\phi:\R^2\to\R^2, (u,v)\mapsto ((u+v)^k(uv^2(u+v)-1),(u+v)(uv^2(u+v)+1).
$$
\end{remark}

A natural question that arise at this point is the following.
\begin{quest}
Which rational functions $f:=\frac{P}{Q}\in\R({\tt u},{\tt v})$ can be soften by means of a surjective regular map $\phi:\R^2\to\R^2$? 
\end{quest}

\begin{remark}
A necessary condition is $Z(Q)\subset Z(P)$ but is is not sufficient as one can check considering the rational function $f(x,y):=\frac{x^2+y^2}{x^4+y^4}$. If we compute the multiplicity of $f\circ\phi$ at any preimage of the origin under a surjective regular function $\phi:=(\phi_1,\phi_2):\R^2\to\R^2$, we observe that it is always negative (because the multiplicity of the denominator doubles the one of the numerator), so $f\circ\phi$ cannot be regular.
\end{remark}

If the following, we will pay a special attention to locally bounded rational functions, that is, rational functions on $\R^2$ that are locally bounded in a neighborhood of its indeterminacy points. As it happens with regulous functions, a locally bounded rational function on $\R^2$ admits only a finite number of poles. 

\begin{lem}\label{lem-bounded} Let $f:=\frac{P}{Q}$ be a locally bounded rational function on $\R^2$ where $P,Q\in\R[\x,\y]$ are relatively prime polynomials. Then the zero-set of $Q$ consists of finitely many points.
\end{lem}

\begin{proof}
As $f=\frac{P}{Q}$ is locally bounded and the polynomials $P,Q\in\R[\x,\y]$ are relatively prime, one has $\{Q=0\}\subset\{P=0\}$. Let $Q_1$ be an irreducible factor of $Q$. Using the criterion \cite[Thm.4.5.1]{bcr} for principal real ideals of $\R[\x,\y]$, we deduce that the ideal $(Q_1)\R[\x,\y]$ is not real. Otherwise, the ideal $\J(\{Q_1=0\})$ of all polynomials of $\R[\x,\y]$ vanishing identically on $\{Q_1=0\}$ is $(Q_1)\R[\x,\y]$, so $Q_1$ divides $P$ against the coprimality of $P$ and $Q$. Consequently, the zero-set of $Q_1$ has by \cite[Thm.4.5.1]{bcr} dimension $\leq 0$, that is, it is a finite set. Applying this to all the irreducible factors of $Q$, we conclude that $\{Q=0\}$ is a finite set, so $f$ has only finitely many poles, as required. 
\end{proof}

As we have already mentioned, regulous functions can be made regular after composition with a finite sequence of blowings-up along smooth algebraic centers \cite{fhmm}. Actually those rational functions on $\R^2$ that become regular after such compositions are exactly the locally bounded rational functions.

\begin{thm}\label{LocBoundBlowUp} Let $f$ be a rational function on $\R^2$. Then $f$ is locally bounded if and only if there exists a finite sequence $\sigma$ of blowings-up along points such that $f\circ \sigma$ is regular.
\end{thm} 
\begin{proof}[Sketch of proof]
Assume first that such a sequence exists. Then the preimage under $\sigma$ of a compact Euclidean neighborhood of an indeterminacy point of $f$ is compact by properness of $\sigma$. Thus, the continuous function $f\circ\sigma$ is bounded on such preimage. As a consequence $f$ is locally bounded at the corresponding indeterminacy point. The converse implication can be proved as \cite[Thm.3.11]{fhmm}. Let us provide an idea on how the proof works. By means of a finite sequence $\sigma:M\to\R^2$ of blowings-up at points, it is possible to make normal crossings the numerator and denominator of $f$. If we choose a suitable local system of coordinates at a point $p\in M$, the rational function $f\circ \sigma$ is equal in such a neighborhood of $p$ to a unit times a quotient of products of the variables (with exponents). As $f\circ \sigma$ is also locally bounded, the denominator must divide the numerator, so $f\circ \sigma$ is regular at $p$, as required.
\end{proof}

\subsection{Double oriented blowings-up}

In the following we restrict our target to some particular type of surjective regular maps: \em compositions of finitely many double oriented blowings-up\em. The oriented blowing-up $\pi^+$ of $\R^2$ at the origin corresponds to the passage to polar coordinates and it can be achieved using the analytic map 
$$
\pi^+:[0,+\infty)\times\sph^1\to\C\equiv\R^2,\ (\rho,v)\mapsto\rho v.
$$
The previous map provides a Nash diffeomorphism between the cylinder $\Cc:=(0,+\infty)\times\sph^1$ and the punctured plane $\Pp:=\R^2\setminus\{(0,0)\}$ whose inverse is $(\pi^+|_\Cc)^{-1}:\Pp\to\Ss,\ u\mapsto(\|u\|,\frac{u}{\|u\|})$. The inverse image of the origin under $\pi^+$ is the circle $\{0\}\times\sph^1$. We choose a rational parameterization of $\sph^1$ (consider for instance the stereographic projection from the North pole) and we obtain the map 
$$
\pi_0:[0,+\infty)\times\R\to\R^2, (\rho,t)\mapsto\Big(\rho\frac{2t}{t^2+1},\rho\frac{t^2-1}{t^2+1}\Big).
$$
The image of this map is $\R^2\setminus\{(0,t):\ t>0\}$. We consider the regular extension $\pi$ of $\pi_0$ to $\R^2$, whose image is $\R^2$ because $\pi(\rho,0)=(0,-\rho)$ for each $\rho\in\R$.

\begin{define}\label{obup} 
The surjective regular map 
$$
\pi:\R^2\to\R^2,\ (\rho,t)\mapsto\Big(\rho\frac{2t}{t^2+1},\rho\frac{t^2-1}{t^2+1}\Big)
$$
is called the \em double oriented blowing-up of $\R^2$ at the origin\em. The double oriented blowing-up of $\R^2$ at an arbitrary point $p\in\R^2$ is the composition $\tau_{\overrightarrow{0p}}\circ\pi\circ\tau_{-\overrightarrow{0p}}$ where $\tau_{\vec{v}}$ denotes the translation $\tau_{\vec{v}}:\R^2\to\R^2,\ x\mapsto x+\vec{v}$ for each $\vec{v}\in\R^2$. Let $\Pi:\C^2\setminus\{t^2+1=0\}\to\C^2$ be the natural rational extension of $\pi$ to $\C^2\setminus\{t^2+1=0\}$.
\end{define}
\begin{lem}
The double oriented blowing-up of $\R^2$ at the origin $\pi:\R^2\to\R^2$ satisfies the following properties:
\begin{itemize}
\item[(i)] $\pi$ is a surjective regular map.
\item[(ii)] $\pi^{-1}((0,0))=\{\rho=0\}$ is the $t$-axis.
\item[(iii)] $\pi|_{\{t=0\}}:\{t=0\}\to\{x=0\},\ (\rho,0)\to (0,-\rho)$ provides a bijection between the $\rho$-axis and the $y$-axis.
\item[(iv)] The determinant of the Jacobian matrix of $\pi$ at the point $(\rho,t)$ values $\frac{2\rho}{t^2+1}$.
\item[(iv)] Both $\pi|_{\R^2\setminus\{\rho=0\}}$ and $\Pi|_{\C^2\setminus\{(t^2+1)\rho=0\}}$ are local diffeomorphisms. 
\item[(v)] Both restrictions $\pi|_{\R^2\setminus\{\rho t=0\}}:\R^2\setminus\{\rho t=0\}\to\R^2\setminus\{x=0\}$ and 
$$
\Pi|:\C^2\setminus\{\rho t(t^2+1)=0\}\to\C^2\setminus\{x(x^2+y^2)=0\}
$$ 
are double covers. In fact, $\pi(-\rho,-\frac{1}{t})=\pi(\rho,t)$ and $\Pi(-\rho,-\frac{1}{t})=\Pi(\rho,t)$.
\end{itemize}
\end{lem}

We can relate the double oriented blowing-up to the classical blowing-up as follows. Denote $\sigma: M\to\R^2$ the blowing-up of $\R^2$ at the origin. We describe $M$ as the subset of $\R^2\times\PP^1$ given by the equation
$$
M:=\{((x,y),[u:v])\in\R^2\times \PP^1:~~xv=yu\}
$$
whereas $\sigma:M\to\R^2,\ ((x,y),[u:v])\mapsto(x,y)$.

\begin{lem}\label{Lem-bup} 
The map 
$$
\psi:\R^2\to M,\ (\rho,t)\mapsto\Big(\Big(\rho\frac{2t}{t^2+1},\rho\frac{t^2-1}{t^2+1}\Big),[2t:t^2-1]\Big)
$$
satisfies the following properties:
\begin{itemize}
\item[(i)] It is a surjective regular map. 
\item[(ii)] The restriction $\psi|_{\{t=0\}}:\{t=0\}\to\Ll:=\{((0,-\rho),[0:1]):\ \rho\in\R\}\subset M$ is bijective.
\item[(iii)] The restriction $\psi|_{\R^2\setminus\{t=0\}}:\R^2\setminus\{t=0\}\to M\setminus\Ll$ is a double cover and it holds $\psi(-\rho,-\frac{1}{t})=\psi(\rho,t)$.
\item[(iv)] $\pi=\sigma\circ\psi$.
\end{itemize}
\end{lem}

\begin{example} 
Consider the locally bounded rational function on $\R^2$ given by the formula $f(x,y):=\frac{x^2}{x^2+y^2}$. Then $(f\circ\pi)(\rho,t)=\frac{4t^2}{4t^2+(t^2-1)^2}$ is regular.
\end{example}

\subsection{Multiplicity of an affine complex curve at a point}
For the sake of completeness we recall how one can compute the multiplicity at a point of a polynomial equation
of a planar curve. 

In the following we use the letter $\omega$ to denote the order of a power series. Let $Q\in\C[\x,\y]$ be a non-constant polynomial and denote $\Cc:=\{Q=0\}\subset\C^2$. Let $p\in\Cc$ be a point of $\Cc$ and assume after a translation that $p$ is the origin. Let $\Gamma_1,\ldots,\Gamma_r$ be the complex branches of $\Cc$ at the origin and let $\gamma_i:=(\gamma_{i,1},\gamma_{i,2})\in\C\{{\tt s}\}^2$ be a primitive parameterization of $\Gamma_i$. We mean with primitive parameterization of $\Gamma_i$ a couple of convergent power series $\gamma_{i,1},\gamma_{i,2}\in\C\{{\tt s}\}$ such that:
\begin{itemize} 
\item both series do not belong simultaneously to the ring $\C\{{\tt s}^k\}$ for any $k\geq 2$,
\item $\Gamma_i$ is the germ at the origin of the set $\{(\gamma_{i,1}(s),\gamma_{i,2}(s)):\ |s|<\veps\}$ for some $\veps>0$ small enough.
\end{itemize}
To compute the multiplicity $\mult_0(Q)$ of $Q$ at the origin, write $Q$ as the sum of its homogeneous components $Q=Q_m+Q_{m+1}+\cdots+Q_d$ where each $Q_k$ is either zero or a homogeneous polynomial of degree $k$ and $Q_m\neq0$. Then $\mult_0(Q)=m$. 

It is possible to express $\mult_0(Q)$ in terms of some invariant associated to the complex branches $\Gamma_1,\ldots,\Gamma_r$ of $\Cc$ at the origin. Write $Q:=P_1^{e_1}\cdots P_s^{e_s}$ where each $P_i\in\C[\x,\y]$ is an irreducible polynomial and $e_i\geq1$. As $\mult_0(Q)=e_1\mult_0(P_1)+\cdots+e_s\mult_0(P_s)$, it is enough to express $\mult_0(P_i)$ in terms of the complex branches of $\Cc$ that correspond to the factor $P_i$. Thus, we assume in the following that $Q$ is an irreducible polynomial.

After a linear change of coordinates, we may assume that $Q$ is a regular series with respect to $\y$ of order $\mult_0(Q)$, that is, $\omega(Q(0,\y))=\mult_0(Q)=m$. By Weierstrass' Preparation Theorem $Q=Q^*U$, where $Q^*\in\C\{\x\}[\y]$ is a distinguished polynomial of degree $m$ and $U\in\C\{\x,\y\}$ is a unit, that is, $U(0,0)\neq0$. Let $Q_1^*,\ldots,Q_\ell^*\in\C\{\x\}[\y]$ be the irreducible factors of $Q^*$ in $\C\{\x\}[\y]$, which are distinguished polynomials with respect to $\y$ such that $\deg_\y(Q_i^*)=\omega(Q_i^*)$. As $Q\in\C[\x,\y]$ is irreducible, it has no multiple factor in $\C\{\x\}[\y]$. Thus, $Q^*=Q_1^*\cdots Q_\ell^*$ and 
$$
\mult_0(Q)=\mult_0(Q^*)+\mult_0(U)=\mult_0(Q^*)=\sum_{j=1}^\ell\mult_0(Q_j^*).
$$
Consequently, it is enough to compute $\mult_0(P)$ for an irreducible distinguished polynomial $P\in\C\{\x\}[\y]$ with $\deg_\y(P)=\omega(P)$. Let $\xi\in\C\{\x^*\}$ be a Puiseux root of $P$. It holds that $\deg_\y(P)$ coincides with the smallest $q\geq 1$ such that $\xi:=\beta(\x^{1/q})\in\C\{\x^{1/q}\}$ for some $\beta\in\C\{\s\}$. In addition, the polynomial $P$ is associated with exactly one complex branch $\Gamma$ for which any primitive parameterization $\gamma:=(\gamma_1,\gamma_2)$ satisfies $q=\min\{\omega(\gamma_1),\omega(\gamma_2)\}$. 

Thus, if we define $\mult(\Gamma)$ as the minimum order of the components of a primitive 
parameterization of $\Gamma$, we have $\mult_0(P)=q=\mult(\Gamma)$.

\subsection{Alternative resolution of indeterminacy}
We know that regulous functions become regular after a finite composition of blowings-up (and this is even true for locally bounded functions as claimed in Theorem \ref{LocBoundBlowUp}). However the source space of the regular function obtained is no longer the same as that of the regulous function we started with. The following result, which is proved below, is an improvement of Lemma \ref{resol0} from which we deduce it.

\begin{thm}\label{resol} 
Let $f$ be a locally bounded rational function on $\R^2$. Then there exists a finite composition $\phi:\R^2\to\R^2$ of finitely many double oriented blowings-up and polynomial isomorphisms such that $f\circ\phi$ is a regular function on $\R^2$.
\end{thm}

The strategy to prove Theorem \ref{resol} is to follow the same process as in classical resolution of indeterminacy \cite[\S5.3]{jp}, but replacing the usual blowing-up along a point by the double oriented blowing-up. The difficulty is to control the number of indeterminacy points, since the double oriented blowing-up is no longer an isomorphism outside the center and it is generically a double cover. A crucial step consists in separating all complex branches passing through an indeterminacy point with different tangents. Note however that, after applying the double oriented blowing-up at an indeterminacy point, the preimage of that point consists of several points each of these with less complex branches than the original indeterminacy point, but possibly with more different tangents!

To face this issue, we state first an auxiliary result that will be needed constantly in the proof of Theorem \ref{resol}.

\begin{lem}\label{align}
Let ${\mathfrak F}:=\{p_1,\ldots,p_n\}\subset\R^2$ be a finite set and let $\Ll'\subset\R^2$ be a line. Pick $i_0\in\{1,\ldots,n\}$ and $q\in\Ll'$. Then there exists a polynomial isomorphism $\varphi:\R^2\to\R^2$ that maps ${\mathfrak F}$ into a half-line in $\Ll'$ issued from $\varphi(p_{i_0})=q$. Moreover, consider
\begin{itemize}
\item an additional line $\Ll$ through $p_{i_0}$,
\item a collection $\{\Hh_1,\ldots,\Hh_r\}$ of complex lines through $p_{i_0}$ different from $\Ll$. 
\item a collection $\{\Hh_1',\ldots,\Hh_s'\}$ of complex lines through $q$ different from $\Ll'$,
\end{itemize}
Then we may assume in addition that the differential of $\varphi$ at $p_{i_0}$ maps $\Ll$ onto $\Ll'$ and the polynomial extension $\Phi:\C^2\to\C^2$ of $\varphi$ to $\C^2$ does not map any of the lines $\Hh_i$ through $p_{i_0}$ into the finite union $\bigcup_{j=1}^s\Hh'_j$ of lines through $q$.
\end{lem}
\begin{proof} 
After an affine change of coordinates, we may assume that $\Ll'$ is the $x$-axis and the restriction to ${\mathfrak F}$ of the projection of $\R^2$ onto the $x$-axis is injective. Write $p_i:=(a_i,b_i)$ and let $P\in\R[\x]$ be an univariate polynomial such that $P(a_i)=b_i$ for $i\neq i_0$ and $P(a_{i_0})\neq b_{i_0}$. Then the polynomial isomorphism $(x,y)\mapsto (x,y-P(x))$ maps the points of ${\mathfrak F}$ except for $p_{i_0}$ into the $x$-axis. Let $\Rr\neq\Ll$ be a line through $p_{i_0}$ non-parallel to the $x$-axis such that the intersection $\Rr\cap\{y=0\}=\{(\lambda,0)\}$ leaves all the points $p_i$ for $i\neq i_0$ in the half-line $\{y\geq\lambda\}$. Consider an affine change of coordinates that keeps the $x$-axis invariant and maps the line $\Rr$ to the $y$-axis. After this change of coordinates $p_{i_0}=(0,\mu_{i_0})$ for some $\mu_{i_0}\neq0$ and $\Ll$ is not parallel to the $y$-axis. Let $Q\in\R[\x]$ be a polynomial such that its graph $\{y-Q(x)=0\}$ passes through the points $p_i$ and whose tangent at the point $p_{i_0}$ is $\Ll$. The polynomial isomorphism $\varphi:\R^2\to\R^2,\ (x,y)\mapsto (x,y-Q(x))$ maps the point $p_{i_0}$ to the origin, keeps the points $p_i$ for $i\neq i_0$ inside the $x$-axis, and the differential of the polynomial isomorphism $\varphi$ at $p_{i_0}$ maps $\Ll$ to the $x$-axis (that is, to the line $\Ll'$). After a translation parallel to the $x$-axis, we may assume in addition that $\varphi$ maps $p_{i_0}$ to $q$. 

As we have much freedom to choose $Q$, we may assume that the polynomial extension $\Phi:\C^2\to\C^2$ of $\varphi$ to $\C^2$ does not map any of the lines $\Hh_i$ through $p_{i_0}$ into the finite union $\bigcup_{j=1}^s\Hh'_j$ of lines through $q$, as required.
\end{proof}

\begin{proof}[Proof of Theorem \em\ref{resol}] 
The set of indeterminacy points of $f$ coincides with the zero-set $\{Q=0\}$, and this is a finite set by Lemma \ref{lem-bounded}. We may assume in addition that $Q$ is non-negative on $\R^2$. Consider the complex algebraic curve $\Cc:=\{z\in\C^2:\ Q(z)=0\}$, which is invariant under complex conjugation in $\C^2$. This means that if $p\in\Cc$ is a real point and $\Gamma$ is a complex branch of $\Cc$ at $p$, then the conjugated branch $\ol{\Gamma}$ is also a complex branch of $\Cc$ at $p$. If $\Tt_p$ is the tangent line to $\Gamma$ at $p$, then its complex conjugated $\ol{\Tt_p}$ is the tangent line to $\ol{\Gamma}$ at $p$. Thus, $\Tt_p=\ol{\Tt_p}$ if and only if $\Tt_p$ admits a linear equation with real coefficients. We are going to solve the indeterminacy points of $f$ by applying a finite chain of double oriented blowings-up centered at the indeterminacy points and polynomial isomorphisms of $\R^2$ like those provided by Lemma \ref{align} (that maps a finite subset of $\R^2$ inside a half-line). 

Denote ${\mathfrak F}:=\{p_1,\ldots,p_n\}$ the set of indeterminacy points of $f$. For each point $p\in{\mathfrak F}$ let $\mult_p(Q)$ be the multiplicity of $Q$ at $p$. Let $p_{i_0}\in{\mathfrak F}$ be such that $M:=\mult_{p_{i_0}}(\Cc)\geq\mult_p(\Cc)$ for each $p\in{\mathfrak F}$. By Lemma \ref{align} there exists a polynomial isomorphism $\varphi:\R^2\to\R^2$ such that if we substitute $f$ by $f\circ\varphi^{-1}$ we may assume that the indeterminacy points of $f$ belong to the negative half $y$-axis $\{(0,y):\ y\leq 0\}$ and $p_{i_0}$ is the origin. In case the germ $\Cc_{p_{i_0}}$ has only one tangent line at the origin, we know that it has to be a real line. If $\Cc$ has a real tangent line at the origin, we may assume in addition by Lemma \ref{align} that this tangent line is the $y$-axis whereas the remaining ones are different from $\x+{\tt i}\y=0$ and $\x-{\tt i}\y=0$. 

Note that the origin is now an indeterminacy point of $f$ with maximal multiplicity and the $y$-axis is one of its tangents. Consider the composition $f\circ\pi$, where $\pi$ denotes the double oriented blowing-up introduced in Definition \ref{obup}. Denote ${\mathfrak F}':=\{p_i:\ i\neq i_0\}$ and ${\mathfrak G}:=\pi^{-1}({\mathfrak F}')$ and let 
$$
\Pi:\C^2\setminus\{t^2+1=0\}\to\C^2
$$
be the natural rational extension of $\pi$ to $\C^2\setminus\{t^2+1=0\}$. As $\pi|_{\R^2\setminus\{\rho=0\}}$ is a local diffeomorphism and $\pi|_{\{t=0\}}:\{t=0\}\to\{x=0\}$ is a bijection between the $\rho$-axis and the $y$-axis, the sets ${\mathfrak G}$ and ${\mathfrak F}'$ have the same number of points and if $q\in{\mathfrak G}$, then $\Pi^{-1}(\Cc)$ has at $q$ the same number of complex branches and different tangents as $\Cc$ has at $\pi(q)$. 

Let us analyze now what happens at the origin. Let $\Gamma_1,\ldots,\Gamma_r$ be the complex branches of $\Cc$ at the origin. The germ $\Cc_0$ equals the union of branches $\bigcup_{i=1}^r\Gamma_i$ and 
$$
M:=\mult_0(Q)=\sum_{i=1}^re_i\mult(\Gamma_i) 
$$
for some positive integers $e_i$. Fix a complex branch $\Gamma_i$ and let $\Tt_i$ be the tangent line to $\Gamma_i$. We have $\Pi^{-1}(\Gamma_i)=\{\rho=0\}_{a_i^1}\cup\{\rho=0\}_{a_i^2}\cup\Lambda_i^1\cup\Lambda_i^2$ where $a_i^1,a_i^2\in\{\rho=0\}\subset\C^2\setminus\{t^2+1=0\}$ and $\Lambda_i^j$ is a complex branch at $a_i^j$ (as we will see in the following sometimes there are exactly two complex branches $\Lambda_i^1, \Lambda_i^2$ and sometimes there is only one and we will consider $\Lambda_i^1=\Lambda_i^2$). As $\pi|_{\{t=0\}}:\{t=0\}\to\{x=0\},\ (\rho,0)\to (0,-\rho)$ and $\pi|_{\R^2\setminus\{\rho t=0\}}:\R^2\setminus\{\rho t=0\}\to\R^2\setminus\{x=0\}$ is a double cover, one has exactly one point $a_i:=a_i^1=a_i^2$ if $\Tt_i$ is the $y$-axis and two different points $a_i^1,a_i^2$ otherwise. Let us analyze the germs $\Lambda_i^j$ for $j=1,2$. Choose a primitive parameterization $\gamma_i:=(\gamma_{i1},\gamma_{i2})$ of $\Gamma_i$. We may assume it has the form 
$$
\begin{cases}
\gamma_{i1}:=\s^{k_i},\\
\gamma_{i2}:=c_i\s^{\ell_i}+\cdots,
\end{cases}\ \text{if $k_i\leq\ell_i$}\quad\text{or}\ 
\begin{cases}
\gamma_{i1}:=c_i\s^{k_i}+\cdots,\\
\gamma_{i2}:=\s^{\ell_i},
\end{cases}\ \text{if $\ell_i<k_i$}
$$ 
with $k_i,\ell_i\geq1$ and $c_i\in\C\setminus\{0\}$. Recall that $\mult(\Gamma_i)=\min\{k_i,\ell_i\}$. Observe that the tangent line to $\Gamma_i$ at the origin is
$$
\Tt_i:=\begin{cases}
\{\y=0\}&\text{if $k_i<\ell_i$,}\\
\{c_i\x-\y=0\}&\text{if $k_i=\ell_i$,}\\
\{\x=0\}&\text{if $k_i>\ell_i$.}
\end{cases}
$$

If we apply the double oriented blowing-up $\pi$ at the origin, we make 
$$
x=\rho\frac{2t}{t^2+1},y=\rho\frac{t^2-1}{t^2+1},
$$
so $\rho^2=x^2+y^2$ and the order of $\rho$ for each complex branch $\Lambda_i^j$ with respect to the variable $\s$ is $\min\{k_i,\ell_i\}$ (recall that $1+c_i^2\neq0$ since the tangent line to $\Gamma_i$ is different from $\x+{\tt i}\y=0$ and $\x-{\tt i}\y=0$). We distinguish several situations:
\begin{itemize}
\item[(i)] If $k_i<\ell_i$, then $\rho=\pm\s^{k_i}+\cdots$ and $t=\pm1+\cdots$. Thus, $\Lambda_i^1$ and $\Lambda_i^2$ are two complex branches at the points $a_i^1:=(0,1)$ and $a_i^2:=(0,-1)$ of multiplicities smaller than or equal to $k_i=\mult(\Gamma_i)$.
\item[(ii)] If $k_i=\ell_i$, then $\rho=\pm\s^{k_i}\sqrt{(1+c_i^2)+\cdots}$. As $1+c_i^2\neq0$, one has $t=(\pm\sqrt{(1+c_i^2)}+c_i)+\cdots$. Thus, $\Lambda_i^1$ and $\Lambda_i^2$ are two complex branches at the points $a_i^1:=(0,\sqrt{(1+c_i^2)}+c_i)$ and $a_i^2:=(0,-\sqrt{(1+c_i^2)}+c_i)$ of multiplicities smaller than or equal to $k_i=\mult(\Gamma_i)$.
\item[(iii)] If $k_i>\ell_i$, then $\rho=-\s^{\ell_i}+\cdots$ and $t=c_i\frac{1}{2}\s^{k_i-\ell_i}+\cdots$. Thus, there is only one complex branch that we write $\Lambda_i:=\Lambda_i^1=\Lambda_i^2$ is a complex branch at the point $a_i:=a_i^1=a_i^2=(0,0)$ of multiplicity smaller than or equal to $\ell_i=\mult(\Gamma_i)$.
\end{itemize}

Starting from the locally bounded rational function $f=P/Q$, we have constructed the rational function $f\circ\pi=P'/Q'$ for some polynomials $P',Q'\in\R[{\tt r},\t]$. As $f\circ\pi$ remains locally bounded, $\{q\in\R^2:\ Q'(q)=0\}$ is a finite set. We have seen above that if two tangent lines $\Tt_i$ and $\Tt_{i'}$ are different, so are the points $a_i^j$ and $a_{i'}^{j'}$ for $j,j'\in\{1,2\}$. Thus,
$$
\mult_{a_i^j}(Q')=\sum_{i':\ \Tt_{i'}=\Tt_i}e_{i'}\mult(\Lambda_{i's}^j)\leq\sum_{i':\ \Tt_{i'}=\Tt_i}e_{i'}\mult(\Gamma_i).
$$
Consequently, if $\Cc$ has more than one tangent line at the origin, 
$$
\mult_{a_i^j}(Q')\leq\sum_{i':\ \Tt_{i'}=\Tt_i}e_{i'}\mult(\Gamma_i)<\sum_{i=1}^re_i\mult(\Gamma_i)=\mult_0(Q)=M
$$
for each point $a_i^j$ and although we have increased the cardinality of the zero set of the new denominator $Q'$, we have dropped the number of real points on which the denominator has multiplicity $M$. 

Assume next that $\Cc$ has only one tangent line at the origin (which is the $y$-axis). As in the usual desingularization process, the multiplicity will decrease but after a finite number of steps. Namely, the unique tangent to $\Cc$ at the origin is $\{\x=0\}$ and $k_i>\ell_i$ (case (iii) above) for $i=1,\ldots,r$. For each complex branch $\Lambda_i$ ($=\Lambda_i^1=\Lambda_i^2$) at the origin, we have the following two possibilities (after reseting the names of the variables and calling $\x$ the first variable and $\y$ the second variable):
\begin{itemize}
\item[(iii.1)] if $k_i':=k_i-\ell_i<\ell_i$, then $\mult(\Lambda_i)<\mult(\Gamma_i)$ and the tangent line to $\Lambda_i$ is $\{\y=0\}$.
\item[(iii.2)] if $k_i':=k_i-\ell_i\geq\ell_i$, then $\mult(\Lambda_i)=\mult(\Gamma_i)$ and we can parameterize the complex branch $\Lambda_i$ by a primitive parameterization $\lambda_i:=(\lambda_{i1},\lambda_{i2})\in\C\{\s\}^2$ such that $\lambda_{i1}:=d_i\s^{k_i'}+\cdots$ and $\lambda_{i2}:=\s^{\ell_i}$. In this case the tangent line is either $\{\x=0\}$ if $k_i'>\ell_i$ or $\{\y-d_i\x=0\}$ if $k_i'=\ell_i$.
\end{itemize}
If situation (iii.1) arises for one of the complex branches of $\Cc$ at the origin, we have $\mult_0(Q')<\mult_0(Q)=M$ and we have dropped the number of real points $p\in\{Q=0\}$ such that $\mult_p(Q)=M$. Otherwise, all the branches $\Lambda_i$ are in situation (iii.2) and $\mult_0(Q')=\mult_0(Q)=M$. The worst situation arises if $\Cc':=\{z\in\C^2:\ Q'(z)=0\}$ has only one tangent line the origin. We can apply to $Q'$ the previous procedure at the origin:
\begin{itemize}
\item We construct (using Lemma \ref{align}) a polynomial isomorphism that maps all the real zeros of $Q'$ into the half line $\{(0,y):\ y\leq 0\}$ and keeps the origin invariant. If $\Cc'$ has a real tangent line, we assume that one of such real tangent lines has equation $\x=0$.
\item We apply the double oriented blowing-up at the origin and repeat the previous discussion. 
\end{itemize}

As the parametrization of the involved complex branches are primitive and $\Lambda_i\cap\R^2=\{(0,0)\}$, if we follow the previous algorithm we realize that the only possibility to get (apparently) stuck in situation (iii.2) for all the complex branches (which is the only case that do not drop the multiplicity) is that $\ell_i=1$ for each $i=1,\ldots,r$. Assume that such is the case and choose a parameterization $\gamma_i:=(\gamma_{i1},\gamma_{i2})$ of the complex branch $\Gamma_i$ (at the origin), that is,
$$
\begin{cases}
\gamma_{i1}:=c_2\s^2+\cdots+c_m\s^m+c_{i,m+1}^*\s^{m+1}+\cdots,\\
\gamma_{i2}:=\s.
\end{cases}
$$
where $c_2,\ldots,c_m\in\R$ (these coefficients are the same for each $i$ because we are `apparently' stuck in situation (iii.2)), $c_{i,m+1}\in\C$ for $i=1,\ldots,r$ and for instance $c_{1,m+1}^*\in\C\setminus\R$ (because the branches of $\Cc$ are all complex as $\Cc\cap\R^2$ is a finite set). Thus, if we apply our algorithm $m+1$ times at the origin, we achieve from the complex branch $\Gamma_1$ a complex branch $\Theta_1$ that admits a primitive parameterization $\theta_1:=(\theta_{11},\theta_{12})\in\C\{\s\}^2$ such that 
$$
\begin{cases}
\theta_{11}:=c_{1,m+1}^*\s+\cdots,\\
\theta_{12}:=\s.
\end{cases}
$$
The tangent line to $\Theta_1$ is $\{\x-c_{1,m+1}^*\y=0\}$ whereas the tangent line to its conjugated complex branch $\ol{\Theta}_1$ is $\{\x-\ol{c}_{1,m+1}^*\y=0\}$, which is different from $\{\x-c_{1,m+1}^*\y=0\}$. Hence, it is not possible to get stuck in situation (iii.2) indefinitely. 

Thus, we are always able to reduce the number of points of multiplicity $M$ with respect to the denominator, using finitely many (polynomial isomorphisms and) double oriented blowings-up. Consequently, after applying suitably the algorithm above finitely many times we find a regular map $\phi:\R^2\to\R^2$ which is a finite composition of (polynomial isomorphisms and) double oriented blowings-up such that $f\circ\phi$ is a locally bounded rational function whose denominator has empty zero-set, that is, $f\circ\phi$ is a regular function, as required.
\end{proof}

\subsection{Arc-lifting property}
As we have seen above any locally bounded rational function on the plane becomes regular after a composition with a generically finite surjective regular map, which can be chosen as a composition of a finite sequence of double oriented blowings-up and polynomial isomorphisms. We have seen also in Example \ref{ex} that other rational functions may becomes regular after composing with a surjective regular map. We propose here to add a geometric condition to the involved surjective regular maps in order to characterize this property.

\begin{define} 
A map $\phi:\R^2\to\R^2$ satisfies the \em arc lifting property \em if for any analytic arc $\gamma:(\R,0)\to (\R^2,p)$, where $p\in\R^2$, there exists an analytic arc $\tilde{\gamma}:(\R,0)\to (\R^2,\tilde{p})$, where $\tilde{p}\in\R^2$, such that $\phi\circ\tilde{\gamma}=\gamma$.
\end{define}

Arc lifting property is a natural condition when dealing with blowing-analytic equivalence \cite{fp}. In particular, a blowing-up along a non-singular center satisfies the arc lifting property, but also a blowing-up along an ideal, or even a real modification \cite[p. 99]{fp}. 

\begin{prop}\label{aa} 
A double oriented blowing-up satisfies the arc lifting property.
\end{prop}

It is possible to use Lemma \ref{Lem-bup} together with the arc lifting property of the blowing-up along a point to prove the arc lifting property of a double oriented blowing-up. However we prefer to produce a direct elementary proof of this fact in order to enlighten the special behavior of a double oriented blowing-up, namely to describe in full details when an arc admits one or two liftings.

\begin{proof}[Proof of Proposition \em \ref{aa}] 
Let $\gamma:=(\gamma_1,\gamma_2):(\R,0)\to (\R^2,p)$ be an analytic arc in $\R^2$ defined on a neighborhood of the origin in $\R$. If we consider the double oriented blowing-up at the origin of $\R^2$, the statement is obvious if $p$ is not the origin. Assume in the following $p:=(0,0)$ and consider $\gamma_1,\gamma_2$ as elements of the ring $\R\{\s\}$ of analytic series in one variable with coefficients in $\R$. By definition of the double oriented blowing-up at the origin we must find analytic series $\rho,t\in\R\{\s\}$ such that 
\begin{equation}\label{dobu}
\Big(\rho\frac{2t}{t^2+1},\rho\frac{t^2-1}{t^2+1}\Big)=(\gamma_1,\gamma_2),
\end{equation}
hence the tuple $\tilde{\gamma}:=(\rho,t)$ satisfies the required conditions. If $\gamma_1=0$, we take $t:=0$ and $\rho:=-\gamma_2$. If $\gamma_2=0$, we take $t:=1$ and $\rho:=\gamma_1$. Thus, we assume in the following $\gamma_1,\gamma_2\neq0$ and write
$$
\gamma_1:=a\s^k+\cdots\quad\text{and}\quad\gamma_2:=b\s^\ell+\cdots
$$
where $a,b\in\R\setminus\{0\}$ and $\ell,k$ are positive integers. Using equation \eqref{dobu} we deduce 
\begin{equation}\label{dobu1}
\rho=\veps\sqrt{\gamma_1^2+\gamma_2^2}\quad\text{and}\quad t=\delta\sqrt{\frac{\rho+\gamma_2}{\rho-\gamma_2}}
\end{equation}
for some $\veps,\delta\in\{-1,+1\}$, as soon as the previous expressions have sense and provide analytic series. We analyze the following three situations:

\par$\bullet$ Assume first that $k<\ell$. Then $\rho^2=\gamma_1^2+\gamma_2^2=a^2\s^{2k}(1+\cdots)$ so that the two possible analytic choices for $\rho$ are given by
$$
\rho:=\veps a\s^k\sqrt{1+\cdots}=\veps a\s^k(1+\cdots),
$$
where $\veps=\pm 1$. Analogously
$$
t^2=\frac{\rho+\gamma_2}{\rho-\gamma_2}=\frac{(\veps a\s^k+\cdots)+(b\s^\ell+\cdots)}{(\veps a\s^k+\cdots)-(b\s^\ell+\cdots)}=1+\cdots.
$$
Therefore we obtain two possible analytic solutions $t:=\delta\sqrt{1+\cdots}=\delta(1+\cdots)$ where $\delta=\pm 1$. The analytic arc $\tilde \gamma=(\rho,t)$ is a lifting of $\gamma$ if equation \eqref{dobu} is satisfied, which is the case if and only if $\veps\delta =+1$. Thus, we have obtained two analytic liftings for $\gamma$.

\par$\bullet$ Assume next $k=\ell$. Proceeding as in the previous case $\rho:=\veps\sqrt{a^2+b^2}\s^k+\cdots$ (with $\veps=\pm1$) whereas 
$$
t^2=\frac{\veps\sqrt{a^2+b^2}+b+\cdots}{\veps\sqrt{a^2+b^2}-b+\cdots}=c_\veps+\cdots
$$
and $c_\veps:=\frac{\sqrt{a^2+b^2}+\veps b}{\sqrt{a^2+b^2}-\veps b}>0$ for both choices of $\veps=\pm 1$, so $t:=\delta\sqrt c_\veps+\cdots$ for some $\delta=\pm1$. Using again equation \eqref{dobu} there exists two analytic liftings of $\gamma$ corresponding to the choice $\veps\delta={\rm sign}(a)$. 

\par$\bullet$ The situation in the remaining case $k>\ell$ is slightly different. The analytic series
$$
\rho:=\veps\sqrt{\gamma_1^2+\gamma_2^2}=\veps|b|\s^\ell+\cdots
$$
for some $\veps=\pm 1$ whereas $t$ must satisfy the equation
\begin{equation}\label{eq-aa}
t^2=\frac{\rho+\gamma_2}{\rho-\gamma_2}=\frac{(\veps|b|+\cdots)+(b+\cdots)}{(\veps|b|+\cdots)-(b+\cdots)}.
\end{equation}
The previous equation has an analytic solution $t$ if and only if $\veps=-{\rm sign}(b)$, that is, $\rho:=-b\s^\ell+\cdots$. We rewrite equation \eqref{eq-aa} as
$$
t^2=\frac{\rho+\gamma_2}{\rho-\gamma_2}=\frac{-{\rm sign}(b)\sqrt{\gamma_1^2+\gamma_2^2}+{\rm sign}(b)\sqrt{\gamma_2^2}}{-{\rm sign}(b)\sqrt{\gamma_1^2+\gamma_2^2}-{\rm sign}(b)\sqrt{\gamma_2^2}}=\frac{\sqrt{1+(\gamma_1/\gamma_2)^2}-1}{\sqrt{1+(\gamma_1/\gamma_2)^2}+1}=\frac{a^2}{4b^2}\s^{2(k-\ell)}+\cdots.
$$
Thus, there exist two possible analytic solutions to \eqref{eq-aa} given by the formula
$$
t:=\delta\frac{a}{2b}\s^{k-\ell}+\cdots
$$
where $\delta=\pm 1$. The couple $(\rho,t)$ provides analytic lifting of $\gamma$ if equation \eqref{dobu} holds and this happens if and only if $\delta=-1$, that is, $t:=-\frac{a}{2b}\s^{k-\ell}+\cdots$. Consequently, in this case we have only one analytic lifting of $\gamma$. 

After the analysis made we conclude that the double oriented blowing-up at the origin satisfies the arc lifting property, as required. 
\end{proof}

The last result allows us to establish that the locally bounded rational functions are exactly those rational functions on $\R^2$ that become regular after composition with a surjective regular map satisfying the arc lifting property (compare this with Theorem \ref{LocBoundBlowUp}).

\begin{thm}\label{SurjectiveArcLift} 
Let $f$ be a rational function on $\R^2$. There exists a surjective regular map $\phi:\R^2 \to\R^2$ satisfying the arc lifting property and such that $ f\circ\phi$ is regular if and only if $f$ is locally bounded.
\end{thm}
\begin{proof}
By Theorem \ref{resol} and Proposition \ref{aa}, it is enough to prove that if $\phi:\R^2\to\R^2$ is a surjective regular map satisfying the arc lifting property and $f\circ\phi$ is a regular function, then $f$ is locally bounded. Otherwise, there exists an analytic arc $\gamma:(\R,0)\to(\R^2,p)$ such that $(f\circ\gamma)(s)$ goes to infinity as $s$ tends to zero. By the arc lifting property there exists an analytic arc $\tilde \gamma:(\R,0)\to (\R^2,\tilde{p})$ such that $\phi\circ\tilde{\gamma}=\gamma$. In particular the analytic arc $f\circ\gamma=f\circ\phi\circ\tilde{\gamma}$ is not bounded at the origin, which is a contradiction.
\end{proof}

\appendix
\section{The open quadrant}\label{s4}

As announced in the Introduction we represent the open quadrant $\Qq:=\{\x>0,\y>0\}$ as the image of simple regulous and regular maps $f,g:\R^2\to\R^2$. 

\renewcommand\thesubsection{\thesection.\arabic{subsection}}
\subsection{A regulous map whose image is the open quadrant}
Consider the regulous map
$$
f:=(f_1,f_2):\R^2\to\R^2,\ (x,y)\mapsto\Big(x^2\Big(\frac{x^2y^2}{x^2+y^2}\Big)+\frac{
(xy+1)^2}{1+(x+y)^2},y^2\Big(\frac{x^2y^2}{x^2+y^2}\Big)+\frac{(xy+1)^2}{1+(x+y)^2}\Big).
$$
We claim: $f(\R^2)=\Qq$.

\begin{proof}
Observe first that $f_1,f_2$ are strictly positive on $\R^2$. Thus, $f(\R^2)\subset\Qq$. Let us prove next the converse inclusion. Pick a point $(a,b)\in\Qq$. Let us show first: \em If $a=b$, then $(a,a)\in g(\R^2)$\em. 

We have
$$
f(t,t)=\Big(\frac{t^4}{2}+\frac{(t^2+1)^2}{1+4t^2},\frac{t^4}{2}+\frac{(t^2+1)^2}{1+4t^2}\Big)
$$
and the image of the previous map contains the half-line $\{(s,s):\ s\geq1\}$. In addition,
$$
f(t,0)=\Big(\frac{1}{t^2+1},\frac{1}{t^2+1}\Big)
$$
and its image contains the segment $\{(s,s):\ 0<s\leq1\}$. Thus, $(a,a)\in f(\R^2)$ for each $a>0$.

If $a\neq b$, we may assume $a<b$. We search $x,y\in\R$ such that $f(x,y)=(a,b)$. Write $y=\lambda x$ where $\lambda\in\R$. We obtain:
\begin{align*}
x^2\Big(\frac{x^2y^2}{x^2+y^2}\Big)+\frac{
(xy+1)^2}{1+(x+y)^2}&=\frac{x^4\lambda^2}{1+\lambda^2}+\frac{
(x^2\lambda+1)^2}{1+x^2(1+\lambda)^2}=a,\\
y^2\Big(\frac{x^2y^2}{x^2+y^2}\Big)+\frac{
(xy+1)^2}{1+(x+y)^2}&=\frac{x^4\lambda^4}{1+\lambda^2}+\frac{
(x^2\lambda+1)^2}{1+x^2(1+\lambda)^2}=b.
\end{align*}
Consequently,
$$
b-a=\frac{x^4\lambda^2(\lambda^2-1)}{1+\lambda^2}\quad\leadsto\quad x=\sqrt[4]{\frac{(1+\lambda^2)(b-a)}{\lambda^2(\lambda^2-1)}}
$$
for some $\lambda>1$ (recall that $a<b$). We deduce
$$
\varphi(\lambda):=\frac{(1+\lambda^2)(b-a)}{\lambda^2(\lambda^2-1)}\frac{\lambda^2}{1+\lambda^2}+\frac{\Big(\sqrt{\frac{(1+\lambda^2)(b-a)}{\lambda^2(\lambda^2-1)}}\lambda+1\Big)^2}{1+(1+\lambda)^2\sqrt{\frac{(1+\lambda^2)(b-a)}{\lambda^2(\lambda^2-1)}}}=a
$$
for some $\lambda>1$. Simplifying we conclude
$$
\varphi(\lambda)=\frac{(b-a)}{\lambda^2-1}+\frac{\Big(\pm\sqrt{\frac{(1+\lambda^2)(b-a)}{(\lambda^2-1)}}+1\Big)^2}{1+\frac{(1+\lambda)^2}{|\lambda|}\sqrt{\frac{(1+\lambda^2)(b-a)}{(\lambda^2-1)}}}
$$
Observe that $\lim_{\lambda\to+\infty}\varphi(\lambda)=0$ while $\lim_{\lambda\to1^+}\varphi(\lambda)=+\infty$. Thence, there exists $\lambda_0>1$ such that $\varphi(\lambda_0)=a$. If we take
$$
x_0:=\sqrt[4]{\frac{(1+\lambda_0^2)(b-a)}{\lambda_0^2(\lambda_0^2-1)}}\quad\text{and}\quad y_0:=\lambda_0x_0,
$$
we have $f(x_0,y_0)=(a,b)$, as required.
\end{proof}
\begin{remark}\label{oq}
Observe that the set of indeterminacy of $f$ is $\{(0,0)\}$ and $f(0,0)=(1,1)$. In addition, $f(\frac{1}{4}\sqrt{-6+2\sqrt{73}},\frac{1}{4}\sqrt{-6+2\sqrt{73}})=(1,1)$. Thus, if $h:\R^2\to\R^2$ is a polynomial map such that $h(\R^2)=\R^2\setminus\{(0,0)\}$ (see Lemma \ref{complement} and \cite[Ex.1.4(iii)]{fg1}), then $g:=f\circ h:\R^2\to\R^2$ is a regular map such that $g(\R^2)=\Qq$. The polynomial map $h:\R^2\to\R^2$ proposed in \cite[Ex.1.4(iii)]{fg1} that has the punctured plane $\R^2\setminus\{(0,0)\}$ as image is given by the formula $(x,y)\mapsto(xy-1,(xy-1)x^2-y)$.
\end{remark}

\subsection{A regular map whose image is the open quadrant}
A simpler regular map $g:\R^2\to\R^2$ than the one proposed in Remark \ref{oq} that has $\Qq$ as its image is provided next. Consider the regular map
$$
g:=(g_1,g_2):\R^2\to\R^2,\ (x,y)\mapsto\Big(x^2(xy-1)^2+\frac{(xy+1)^2}{1+(x+y)^2},y^2(xy-1)^2+\frac{(xy+1)^2}{1+(x+y)^2}\Big).
$$
We claim: $g(\R^2)=\Qq$.

\begin{proof} 
We proceed analogously to the preceding case, but replacing the lines $y=\lambda x$ by hyperbolas $y=\lambda/x$. As $g_1,g_2$ are strictly positive on $\R^2$, we have $g(\R^2)\subset\Qq$. To prove the converse inclusion pick a point $(a,b)\in\Qq$. Let us show first: \em If $a=b$, then $(a,a)\in g(\R^2)$\em. 

We have
$$
g(t,t)=\Big(t^2(t^2-1)^2+\frac{(t^2+1)^2}{1+4t^2},t^2(t^2-1)^2+\frac{(t^2+1)^2}{1+4t^2}\Big).
$$
The image of $g(t,t)$ contains the half-line $\{(s,s):\ s\geq \frac{4}{5}\}$ (for $t=1$ we have $g(1,1)=(\frac{4}{5},\frac{4}{5})$). In addition,
$$
g(t,1/t)=\Big(\frac{4}{1+(t+1/t)^2},\frac{4}{1+(t+1/t)^2}\Big)=\Big(\frac{4t^2}{t^2+(t^2+1)^2},\frac{4t^2}{t^2+(t^2+1)^2}\Big)
$$
and its image contains the segment $\{(s,s):\ 0<s\leq \frac{4}{5}\}$ (for $t=1$ we have $g(1,1)=(\frac{4}{5},\frac{4}{5})$). Consequently, $(a,a)\in g(\R^2)$.

If $a\neq b$ we may assume $a>b$. We search $x,y\in\R$ such that $g(x,y)=(a,b)$. Write $y=\lambda /x$ for some $\lambda\in\R$. Then
$$
a-b=g_1(x,\lambda/x)-g_2(x,\lambda /x)=(\lambda-1)^2(x^2-\lambda^2/x^2)
$$
Consequently, for $\lambda \neq 1$
$$
x^4-\frac{a-b}{(\lambda-1)^2}x^2-\lambda^2=0 \quad\leadsto\quad x^2=r(\lambda):=\Big(\frac{a-b}{(\lambda-1)^2}+\sqrt{\frac{(a-b)^2}{(\lambda-1)^4}+4\lambda^2}\Big)\Big/2.
$$
We have $\lim_{\lambda\to+\infty}r(\lambda)=+\infty$ and $\lim_{\lambda\to1^+} (\lambda-1)^2r(\lambda)=a-b$.

Consider the equation $f_1(x,\lambda/x)=a$, that is,
$$
(\lambda-1)^2x^2+\frac{(\lambda+1)^2x^2}{x^2+(x^2+\lambda)^2}=a,
$$
and replace $x^2$ by $r(\lambda)$ to obtain
$$
\varphi (\lambda):=(\lambda-1)^2r(\lambda)+\frac{(\lambda+1)^2r(\lambda)}{r(\lambda)+(r(\lambda)+\lambda)^2}=a.
$$
Observe that $\lim_{\lambda\to+\infty}\varphi(\lambda)=+\infty$ while $\lim_{\lambda\to1^+}\varphi(\lambda)=a-b$. As $a>a-b$, there exists $\lambda_0>1$ such that $\varphi(\lambda_0)=a$. If we take
$$
x_0:=\sqrt{\Big(\frac{a-b}{(\lambda_0-1)^2}+\sqrt{\frac{(a-b)^2}{(\lambda_0-1)^4}+4\lambda_0^2}\Big)\Big/2}\quad\text{and}\quad y_0:=\lambda_0/x_0,
$$
we have $g(x_0,y_0)=(a,b)$, as required.
\end{proof}

\bibliographystyle{amsalpha}

\end{document}